\documentclass{article}

\usepackage{graphicx}
\usepackage{latexsym}
\usepackage{amssymb}
\usepackage{amsmath, amsthm, amsfonts}
\usepackage{graphicx}
\usepackage{color}
\usepackage{xcolor}
\usepackage{latexsym}
\usepackage{multirow}
\usepackage{authblk}

\def\Underline{\setbox0\hbox\bgroup\let\\\endUnderline}
\def\endUnderline{\vphantom{y}\egroup\smash{\underline{\box0}}\\}
\def\|{\verb|}

\usepackage[varg]{txfonts}
\makeatletter%
\input{ot1txtt.fd}
\makeatother%

\theoremstyle{plain}
\newtheorem{theorem}{Theorem}[section]
\newtheorem{remark}[theorem]{Remark}
\newtheorem{proposition}[theorem]{Proposition}
\newtheorem{lemma}[theorem]{Lemma}
\newtheorem{corollary}[theorem]{Corollary}

\theoremstyle{definition}
\newtheorem{definition}[theorem]{Definition}
\numberwithin{equation}{section}
\numberwithin{table}{section}
\numberwithin{figure}{section}

\newtheorem{axiom}{Axiom}
\newtheorem{example}{Example}

\newcommand{\sL}{*L}
\newcommand{\sR}{*R}
\newcommand{\iso}{=}
\newcommand{\eq}{\equiv}
\newcommand{\oL}{\mathcal{L}}
\newcommand{\oR}{\mathcal{R}}
\newcommand{\oP}{\mathcal{P}}
\newcommand{\oN}{\mathcal{N}}
\newcommand{\fzz}{\mid \mid}
\newcommand{\conj}[1]{\overline{#1}}
\newcommand{\kome}{\maltese}
\newcommand{\stap}{\bigstar}

\newcommand{\LR}{\mathcal{L}\mathcal{R}}

\begin{document}

\title{$\LR-$Ending Partisan Rulesets}

\author[1]{Hiroki Inazu\thanks{Email: \texttt{d220585@hiroshima-u.ac.jp}}}
\author[1]{Shun-ichi Kimura\thanks{Email: \texttt{skimura@hiroshima-u.ac.jp}}}
\author[2,3,4]{Koki Suetsugu\thanks{Email: \texttt{suetsugu.koki@gmail.com}}}

\affil[1]{Hiroshima University}
\affil[2]{Osaka Metropolitan University}
\affil[3]{Waseda University}
\affil[4]{Toyo University}

\date{}

\maketitle

\begin{abstract}
In this paper, we consider $\LR$-ending partisan rulesets as a branch of combinatorial game theory. In these rulesets, the sets of options of both players are the same. However, there are two kinds of terminal positions. If the game ends in one kind of terminal position, then a player wins, and if the game ends in the other kind of terminal position, the other player wins. We introduce  notations for positions in $\LR$-ending partisan rulesets including disjunctive sum of terminal positions depending on the parity and basic definitions and show their algebraic structures. We also introduce some examples of $\LR$-ending partisan rulesets and show how our results can be used for analyzing the rulesets.
\end{abstract}

\section{Introduction}
\label{sec:intro}
Combinatorial game theory studies algebraic structures in two-player games with no chance moves nor hidden information. 
Traditionally, impartial games (games in which the sets of options for both players are the same) and partisan games (games in which the sets of options for both players can be different) under the normal play convention (the player who moves last is the winner), which are known to have relatively simple algebraic structures,  are mainly studied. 

However, in addition to them, recently, games under various conventions are investigated. 
For example, 
mis\`{e}re games studied in \cite{Pl05,Pl09} and guaranteed scoring games studied in \cite{LNNS16}.
This paper also introduces a new convention and shows algebraic structures of positions under the convention.

In \cite{EP}, games with a new method to determine the winner was proposed, which is called ending partisan games.  There are multiple terminal positions in the game, and the winner is determined by which terminal position the game ends, not by who plays last, nor who got more scores.   One important (and the original) example is the $\LR$-ending partisan subtraction game(or $\LR$-ending partisan nim).  This game is played by two players, say Left (female) and Right (male), with the given set of removable numbers $S\subset \mathbb{Z}_{\ge 2}$, where $s_1 = \min S$. For an integer $n$, $\mathbb{Z}_{\ge n}$ is the set of integers greater than or equal to $n$.  In the original paper, the game positions are represented by a number of tokens $a\in \mathbb{Z}_{\ge 0}$ in a pile, and the player can take $s\in S$ tokens from the pile.  When $a<s_1$, the game ends, and the winner is determined by the parity of $a$.  When $a$ is an even number,  Left wins, and when $a$ is an odd number, Right wins. 
In \cite{HI}, the first author of this paper generalizes this game to a multi-pile situation: the game positions are represented by $(a_1, \ldots, a_n)$, such that the $i$-th pile has $a_i$ tokens, and the player chooses one pile, and takes $s\in S$ tokens from the pile.  Winner is determined by the total number of remaining tokens similarly to the one-pile situation, namely, Left wins if the total number of remaining tokens is even, and Right wins otherwise.  Notice that in terms of options, this is exactly the disjunctive sum of one-pile positions, and in this paper, the notion of disjunctive sum is based on this winner-decision rule. One subtle, but decisively different method to determine the winner is that, at the end of the game, count the number of piles with even number of tokens and odd number of tokens, and if there are more even numbered piles, Left is declared as the winner, and if there are more odd numbered piles, Right is declared as the winner.  But surprisingly, it is known that this alternative ruleset is pretty unfair.
In \cite{EP}, it was observed that even though the winner decision scheme seems fair to both players, when there is only one pile, Left has a big advantage over Right: For most of the choices of $S$, Left has a winning strategy, both as the first player or the second player, as far as the original number of tokens is large enough.  For example when $S=\{2, 3, 6\}$ and there are more than 7 tokens, Left has the winning strategy, whoever plays first.  If we use the alternative ruleset, to decide the winner by counting the even numbered piles, this unfairness becomes even more pronounced, and it would be very difficult for Right to win the game.  We believe that our ruleset to count the parity of the total remaining tokens is much fairer, since as shown in Proposition \ref{prof:plussr}, changing the parity of the total number of remaining tokens reverses the advantage between Left and Right.

The following is an example of playing this game:  
$$(6, 11, 2) \xrightarrow{\text{Left}} (6, 5, 2) \xrightarrow{\text{Right}}(6, 5, 0) \xrightarrow{\text{Left}} \cdots \xrightarrow{\text{Right}}(1, 1, 0)$$
In the first move, Left removes six tokens from the second pile. 
When the game proceeds as described above, it ends at position (1,1,0) since from every pile no one can remove two, three or six tokens.  Then, the total number of remaining tokens is two, which is an even number, and Left becomes the winner. This game can also be interpreted as follows.
Suppose there is a single light controlled by three independent switches, and the numbers $(6,11,2)$ represent
how many times each switch can still be toggled.
Initially, the light is turned on.
In each move, a player must toggle one switch exactly two, three or six times.
When no switch can be toggled any further, the game ends.
If the light is off at the end of the game, then Left wins; otherwise, Right wins.
Since such games do not fit within the frameworks examined, we introduce a new one in this paper.

For normal play or mis\`ere play, all terminal positions are identified as the same position, say $0=\{|\}$, and its disjunctive sum satisfies $0+0=0$.  In other words, the terminal positions form a trivial group $(0, +)$, where normal play considers $0$ as a $\oP$-position(last player wins) and mis\`ere play considers $0$ as an $\oN$-position(last player loses).  For scoring games, there exist infinitely many terminal positions, say $\{\varnothing^a \, |\, a\in \mathbb{R}\}$, where $\varnothing^a$ is more favorable for Left than $\varnothing^b$ when $a>b$.  In this case, the terminal positions form a group $(\mathbb{R}, +)$.  From this viewpoint, our paper proposes a new group of terminal positions $(\mathbb{Z}/2\mathbb{Z}, +)=(\{*L, *R\}, +)$, which is one of the simplest examples of the (semi)-group of terminal positions next to the normal play and the mis\`ere play.

The organization of this paper is as follows.
In Section \ref{sec:def}, we introduce a notation for positions in $\LR$-ending partisan rulesets and prove basic theorems. In Section \ref{sec:values}, we discuss some game values and their structures. Furthermore, in Section \ref{sec:examples}, we show some rulesets which can be analyzed by results in Section \ref{sec:values}.
Finally, in Section \ref{sec:open}, we show some open problems.

\section{Basic Definitions and Lemmas}
\label{sec:def}

In this section, we introduce basic concepts in $\LR$-ending partisan rulesets and prove several theorems that will be used in the subsequent sections.
The concept of $\LR$-ending partisan rulesets is new in combinatorial games, but there are similar notions to classical games and they are given the same names as in the classical setting. For the classical theories of combinatorial game theory, see \cite{CGT,ANW}.

In classical studies of combinatorial game theory, there are two famous winning conventions. One is the normal play convention and the other is the mis\`{e}re play convention. In normal (resp. mis\`{e}re) play convention, the player who moves last is the winner (resp. loser).
Therefore, there is only one kind of terminal position and the winner depends on which player makes the final move.
However, we can consider another convention. For example, there are two kinds of terminal positions and the winner depends on which kind of terminal positions the final move reaches as shown in Section \ref{sec:intro}.

A ruleset is an {\em $\LR$-ending partisan ruleset} if the options of positions for both players are the same and there are two kinds of terminal positions. If the game ends in one kind of them, Left becomes the winner, and if the game ends in the other kind of terminal positions, Right becomes the winner. For fairness, we determine the disjunctive sum of terminal positions by the parity.
That is, if there are an even number of Right's terminal positions, then the winner is Left, and if there are an odd number of Right's terminal positions, then the winner is Right.
This setting matches the original setting for $\LR$-ending partisan subtraction game.

\subsection{Definition of positions}
In $\LR$-ending partisan rulesets, similar to the notation in \cite[Chapter 11]{C01}, for impartial rulesets, 
 the positions are defined as follows:

\begin{definition} \mbox{}
\label{def:positions}
    \begin{itemize}
        \item $\sL$ and $\sR$ are (terminal) positions, in which Left and Right are the winners, respectively.
        \item For $n > 0$, if $G_1, \ldots, G_n$ are positions, then $\{G_1, \ldots, G_n\}$ is also a position.
    \end{itemize}
    
\end{definition}
In the latter case, $G_i$ is called an {\em option} of $\{G_1, \ldots, G_n\}$. 

\begin{definition} \mbox{}
\label{def:isomorphism}
Let $G$ and $H$ be positions in $\LR$-ending partisan rulesets. Then, we recursively define {\em $G$ is isomorphic to $H$} as follows:

$G$ is isomorphic to $\sL$ if and only if $G$ is a terminal position and Left wins. Similarly, $G$ is isomorphic to $\sR$ if and only if $G$ is a terminal position and Right wins.

When $G$ has options $G_1, \dots, G_n$ and $H$ has options $H_1, \dots, H_m$, then $G$ is isomorphic to $H$ if and only if $n = m$ and there is a bijection $f: \{G_1, \dots, G_n\} \to \{H_1, \dots, H_m\}$ such that $G_i$ is isomorphic to $f(G_i)$.
    When $G$ is isomorphic to $H$, we denote it as $G \iso H$.
\end{definition}

In other words, if
the game trees of two positions are the same and there exists a one-to-one correspondence between their terminal positions that preserves their kinds, the two positions are 
isomorphic.

Similarly to classical combinatorial game theory, we say the {\em birthday} of a position is the height of the game tree of the position. The birthday of a position $G$ is denoted as $b(G)$.
\begin{definition}[Birthday]
    $b(\sL) = b(\sR) = 0$.
    For any position $G \iso \{G_1, \ldots, G_n\}, b(G) = \max(\{b(G_i)\}) + 1$. 
\end{definition}

A position is {\em born on day $d$} if $b(G) = d$. Also, a position is {\em born by day $d$} if $b(G) \leq d$. Note that we assume that every position is short, that is, for every position $G$, there is a nonnegative integer $d$ such that $b(G) = d$. We also note that  ``induction on $G$'' means induction on the birthday of $G$.

\subsection{Disjunctive sum}
Next, we introduce the disjunctive sum of positions in $\LR$-ending partisan rulesets. 

For the base cases of the recursive definition of the disjunctive sum, we introduce the following axiom.

\begin{axiom}\mbox{}
\label{axm:base}
        $$\sL + \sL \iso \sR + \sR \iso \sL.$$
        $$\sL + \sR \iso \sR + \sL \iso \sR.$$
\end{axiom}

\begin{remark}
    In piles and tokens games such as subtraction games, for multi-pile terminal positions, the winner is determined by the parity of the total number of remaining tokens. For example, in the $\LR$-ending partisan subtraction game with $S = \{2, 5\}$, a position $(1,1)$ is a terminal position, and since there are 2 tokens, Left wins. From Definition \ref{def:positions} and Axiom \ref{axm:base}, we also have $\sR + \sR = \sL$.
\end{remark}

\begin{definition}\label{def:sum}
    Let $G$ and $H$ be positions in $\LR$-ending partisan rulesets and we assume that at least one of $G$ or $H$ is not a terminal position. 
    Then, the disjunctive sum of $G$ and $H$ is denoted by $G + H$
    and is defined recursively as follows:

    If $G \iso \{G_1, \ldots, G_n\}, H \iso \{H_1, \ldots, H_m\}$, then
    $$G + H = \{G_1 + H, \ldots, G_n + H, G + H_1, \ldots, G + H_m \}.$$
\end{definition}

Note that if either $G$ or $H$ is a terminal position, then corresponding options do not appear. For example, when $G = \sL$ or $\sR$, we write $G = \{\}$ (empty set) in the notation above.

\begin{proposition}
    For any positions $G, H$ and $J$, the following hold:
    $$G + H \iso H + G.$$
    $$ (G + H) + J \iso G + (H + J).$$
\end{proposition}
\begin{proof}
Let $G \iso \{G_1, \ldots, G_n\}, H \iso \{H_1, \ldots, H_m\},$ and $J \iso \{J_1, \ldots, J_\ell\}$.
We prove this by induction on the sum of the birthdays of $G, H, $ and $J$.
If $G, H, $ and $J$ are terminal positions, both commutativity and associativity are obvious from Axiom \ref{axm:base}.

Assume that $b(G) + b(H) > 0$.
Then, from the induction hypothesis, 
\begin{eqnarray*}
    G + H &\iso& \{G_1 + H, \ldots, G_n + H, G + H_1, \ldots, G + H_m \} \\ 
    &\iso& \{H + G_1, \ldots, H+G_n, H_1 + G, \ldots, H_m + G\} \\ &\iso&
    H + G.
\end{eqnarray*}

Assume that $b(G) + b(H) + b(J) > 0$. Then, from the induction hypothesis,

\begin{align*}
&(G + H) + J\\
\iso& \{G_1 + H, \ldots, G_n + H, G + H_1, \ldots, G + H_m \} + J \\ 
\iso& \{(G_1 + H) + J,\ldots, (G_n + H) +J, (G+ H_1) + J,\\
&\ldots, (G + H_m) + J, (G + H) + J_1, \ldots, (G +H) + J_\ell\} \\
\iso& \{G_1 + (H + J), \ldots, G_n + (H +J), G+ (H_1 + J),\\
&\ldots, G + (H_m + J), G + (H + J_1), \ldots, G +(H + J_\ell)\} \\
\iso& G + \{H_1 + J, \ldots, H_m + J, H + J_1, \ldots, H + J_\ell \} \\
\iso& G+ (H+ J).
\end{align*}
\end{proof}

\subsection{Outcome classes}
We show that the set of positions can be separated into four sets depending on which player has a winning strategy.

\begin{definition}[Outcome classes]\label{def:outcome} \mbox{}

Let $G$ be a position. Then, $G$ is exactly one of the four possibilities below:

  \begin{enumerate}
    \item[$(\oL)$] starting from $G$, regardless of who is the first player, Left has a winning strategy.
    \item[$(\oR)$] starting from $G$, regardless of who is the first player, Right has a winning strategy.
    \item[$(\oN)$] starting from $G$, regardless of who is the first player, the first player (or the Next player) has a winning strategy.
    \item[$(\oP)$] starting from $G$, regardless of who is the first player, the second player (or the Previous player) has a winning strategy.
  \end{enumerate}

 We write $o(G)$ as one of $\oL, \oR, \oN, \oP$ accordingly and call it the {\em outcome} of $G$.
 Since no two of $\oL, \oR, \oN,$ and $\oP$ can occur simultaneously, a set of positions is partitioned into four sets.
When no confusion arises, we write $G \in \oL$ (respectively, $\oR, \oN, \oP$) to mean that $o(G)=\oL$ (respectively, $\oR, \oN, \oP$).
\end{definition}

For any position, the name of its outcome corresponds to the player who has a winning strategy in the position.

\begin{theorem}\mbox{}
\label{thm:outcome}
For any position $G$, outcome can be mathematically characterized as follows:
\begin{itemize}
    \item $\sL \in \oL, \sR \in \oR$. 
    \item Consider a position $G \iso \{G_1, \ldots, G_n\}$,
\begin{eqnarray*}
    G \in \oN &\text{if and only if}& \begin{array}{l}\text{There exist $G_i$ and $G_j$ such that} \\ \text{$G_i \in \oL \cup \oP$ and $G_j \in \oR \cup \oP$.} \end{array}\\
    G \in \oL &\text{if and only if}& \begin{array}{l}\text{There exists $G_i $ such that $G_i \in \oL$ } \\ \text{and for any $G_j, G_j \in \oL \cup \oN$.}\end{array} \\
    G \in \oR &\text{if and only if}& \begin{array}{l}\text{There exists $G_i $ such that $G_i \in \oR$ } \\ \text{and for any $G_j, G_j \in \oR \cup \oN$.}     \end{array} \\
    G \in \oP &\text{if and only if}& \begin{array}{l} \text{For any $G_i$, $G_i \in \oN$.}\end{array}
\end{eqnarray*}
\end{itemize}
\end{theorem}
\begin{proof}
    We proceed by induction on $G$.  
    If $G = \sL$, then it is a terminal position and Left wins. Thus, $\sL \in \oL$. Similarly, $\sR \in \oR$.    
    
    We assume that $G \iso \{G_1, \ldots, G_n\}$ and that for each $G_i$, $o(G_i)$ is already determined.
    
    If there exist $G_i$ and $G_j$ such that $G_i \in \oL \cup \oP$ and $G_j \in \oR \cup \oP$, Left wins as the first player by moving to $G_i$ and Right wins as the first player by moving to $G_j$. Hence, the first player wins, and it follows that $G \in \oN$.    
    
    If there exists $G_i $ such that $G_i \in \oL$ and for any $G_j, G_j \in \oL \cup \oN$, then Left wins as the first player by moving to $G_i$. If Right moves to $G_j$, then since $G_j \in \oL \cup \oN$, Left wins as the first player at $G_j$. Thus, Left has a winning strategy for $G$, and it follows that $G \in \oL$.    
    
    Similarly, If there exists $G_i$ such that $G_i \in \oR$ and for any $G_j$, $G_j \in \oR \cup \oN,$ then $G \in \oR$.     
    
    Finally, if for any $G_i$, $G_i \in \oN,$ then the second player at $G$ has a winning strategy for every $G_i$ as the first player. Hence, $G \in \oP$.
    
    Since $\oL, \oR, \oN$ and $\oP$ are pairwise disjoint, the converse implication also holds.
\end{proof}

\begin{proposition}
\label{prof:plussr}
For any position $G,$ the following hold:
    \begin{enumerate}
        \item $G \in \oL$ if and only if $G + \sR \in \oR$.
        \item $G \in \oR$ if and only if $G + \sR \in \oL$.
        \item $G \in \oP$ if and only if $G + \sR \in \oP$.
        \item $G \in \oN$ if and only if $G + \sR \in \oN$.
        
    \end{enumerate}
\end{proposition}

\begin{proof}
    We proceed by induction on $G$. If $G$ is a terminal position, then, by Axiom \ref{axm:base}, we have $*L + *R = *R \in \oR$, $*R + *R = *L \in \oL$.
    We assume that $G = \{G_1, \dots, G_n\}$ is not a terminal position and that every $G_i$ satisfies the statement.

    If $G \in \oL$, then there exists $G_i$ such that $G_i \in \oL$ and for any $G_j$, $G_j \in \oL \cup \oN$, by Theorem \ref{thm:outcome}.
    By the induction hypothesis and $G + *R = \{G_1 + *R, \dots, G_n + *R\}$, it follows that $G_i + *R \in \oR$ and for any $G_j$, $G_j + *R \in \oR \cup \oN$ holds. Thus, $G + *R \in \oR$.

    Similarly, if $G \in \oR$, then $G + *R \in \oL$.

    If $G \in \oN$, then there exist $G_i$ and $G_j$ such that $G_i \in \oL \cup \oP$ and $G_j \in \oR \cup \oP$, by Theorem \ref{thm:outcome}.
    By the induction hypothesis, we have $G_i + *R \in \oR \cup \oP$ and $G_j + *R \in \oL \cup \oP$. It follows that $G + *R \in \oN$.

    If $G \in \oP$, then for any $G_i, G_i \in \oN$ holds.
    By the induction hypothesis, we have $G_i + *R \in \oN$. It follows that $G + *R \in \oP$.

    Since $\oL, \oR, \oN, \oP$ are pairwise disjoint, the converse implication also holds.
\end{proof}

In $\LR$-ending partisan game, when $o(G) = \oL$, Left has an advantage in $G$, but as shown in this proposition, in general (e.g., $G + \sR$), it may be an advantage for Right. From this perspective, we have another interpretation of the outcomes, in addition to that given in Definition \ref{def:outcome}. $o(G) = \oL$ (resp. $o(G) = \oR$) means that regardless of which player plays first, they have a strategy in which the game ends at $\sL$ (resp. $\sR$). Similarly, $o(G) = \oN$ (resp. $o(G) = \oP$) means that the first player (resp. the second player) has a strategy in which they can choose either $\sL$ or $\sR$ as the terminal position.

\subsection{Equivalence relation}
 It is clear that isomorphism of positions is an equivalence relation. However, in combinatorial game theory, a weakened equivalence relation is more frequently used. That is, if  swapping $G$ and $H$ does not change the outcome class whatever position is added to both positions, then $G$ and $H$ can be considered equivalent. In this 
 paper, we also define equivalence in this manner.

\begin{definition}[Equivalent]
\label{def:equiv}
If positions $G$ and $H$ satisfy $o(G+X) = o(H+X)$ for any  position $X$, then we say $G$ and $H$ are {\em equivalent}, or $G \eq  H$.     
\end{definition}

\begin{remark}
    As we adopt Conway Style notation \cite[Chapter 0]{C01},
 the positions of $\LR$-ending partisan ruleset, defined in Definition \ref{def:positions} form a universal $\LR$-ending partisan ruleset, which satisfy Axiom 1, with no other terminal positions than $\sL$ and $\sR$.  In Definition \ref{def:equiv}, X runs over all positions in this universal ruleset. 
\end{remark}

\begin{proposition}
    The relation $\eq$ is an equivalence relation. Moreover, if $G \eq H$, then for any position $K$, $G + K \eq H + K$ holds.
\end{proposition}
\begin{proof}
It is easy to confirm that $G \eq G, G \eq H \Longrightarrow H \eq G,$ and $(G \eq H) \land  (H \eq J) \Longrightarrow G \eq J$ from Definition \ref{def:equiv}. 
Since for any position $X$, $K + X$ is also a position, if $G \eq H$, then $o(G + K + X) = o(H + K + X)$ holds. It follows that $G + K \eq H + K$.
\end{proof}

For this equivalence relation, in classical documents on combinatorial game theory, the symbol $=$ is used. However, recently, the symbol $\equiv$ has been used in some papers such as \cite{UN}. This is to avoid confusion, and we also follow this notation. 

\subsection{Conjugate}
In partisan games under normal play convention, a position in which the roles of Left and Right are swapped from the original position becomes the additive inverse of the original position. However, except for normal play convention, the existence of an additive inverse is not guaranteed in general. In such conventions, a position in which the roles of Left and Right are swapped is called a  conjugate of the original position. In $\LR$-ending partisan rulesets, as we will show in Corollary \ref{cor:noinverse}, there are some positions that do not have an additive inverse. Therefore, we will call swapped positions {\em conjugates}.

\begin{definition}[Conjugate] \mbox{}
When $G$ is a position, we define its conjugate, denoted as $\conj{G}$, inductively as follows:
\begin{itemize}
    \item $\conj{\sL} \iso \sR, \conj{\sR} \iso \sL$.
    \item For any position $G \iso \{G_1, \ldots, G_n\}, \conj{G} \iso \{\conj{G_1}, \ldots, \conj{G_n} \}. $
\end{itemize}
\end{definition}

\begin{example}
    In $\LR$-ending partisan subtraction game with $S = \{2, 5\}$, let $G = (2) = \{\sL\}$. Then its conjugate is $\conj{G} = (3) = \{\sR\}$.
\end{example}

\begin{proposition}
\label{prop:rconj}
For any position $G$,
    $G + \sL \iso G$ and $G + \sR \iso \conj{G}$ hold.
\end{proposition}

\begin{proof}
    We prove this by induction on 
    $G$. If $G \iso \sL, $ then we have $ \sL + \sL \iso \sL $ and $\sL + \sR \iso \sR$ from Axiom \ref{axm:base}. Also, if $G \iso \sR, $ then we have $\sR + \sL \iso \sR$ and $\sR + \sR \iso \sL$. 

    Assume that $G \iso \{G_1, \ldots, G_n\}$ is not a terminal position. Then, from the induction hypothesis, we have
    $$G + \sL \iso \{G_1 + \sL, \ldots, G_n + \sL\} \iso \{G_1, \ldots, G_n\} \iso G,$$
    and 
    $$G + \sR \iso \{G_1 + \sR, \ldots, G_n + \sR\} \iso \{\conj{G_1}, \ldots, \conj{G_n}\} \iso \conj{G}.$$    
\end{proof}

\begin{proposition}
\label{prop:conj}
\mbox{}
For any position $G$, the following statements hold.
\begin{enumerate}
    \item $\conj{\conj{G}} \iso G$.
    \item $G + G \in \oL \cup \oP$.
    \item $G + \conj{G} \in \oR \cup \oP$.
\end{enumerate}
\end{proposition}

\begin{proof}
(1) This is from the definition of conjugate and induction on $G$.

(2) It is enough to show that Left can win from $G + G$ as the second player. We prove this by induction on $G$. If $G$ is a terminal position, then $G + G \in \oL$ since $\sL + \sL \iso \sR + \sR \iso \sL \in \oL$. If $G$ is not a terminal position, then for any option $G_i + G$, there is an option $G_i + G_i \in \oL\cup \oP$ from the induction hypothesis. Therefore, $ G + G \in \oL \cup \oP$.

(3) The proof is similar to (2).
\end{proof}

All positions satisfy Proposition \ref{prop:conj}(1), but some positions satisfy the condition $\conj{G} = G$. The following Proposition describes some properties of such positions.

\begin{proposition}
\label{prop:np}
    If $G \iso \conj{G}$, then $G + G \in \oP$ and for any position $H$, $G + H \in \oN \cup \oP$ hold.
\end{proposition}
\begin{proof}
    From Proposition \ref{prop:conj} (2) and (3), we have $G + G \in \oL \cup \oP$ and $G + G \in \oR \cup \oP$. Therefore, $G + G \in \oP$. 
    
    For any position $H$, $\conj{G + H} = G + H + \sR = \conj{G} + H = G + H$.
    $G + H \iso \conj{G + H}$ implies that $o(G + H) = o(G + H + \sR)$, from Proposition \ref{prof:plussr}, it follows that $G + H \in \oN \cup \oP$.
\end{proof}

\begin{corollary}
    \label{cor:noinverse}
    If $G = \conj{G}$, then for any $G', G + G' \not \equiv \sL$. 
\end{corollary}

\begin{proof}
    By Proposition \ref{prop:np}, for any position $G'$, we have $G + G' \in \oN \cup \oP$. Since $\sL \in \oL$, it follows that $G + G' \not \equiv \sL$.     
\end{proof}

\subsection{Partial order}
In partisan rulesets under normal play convention, 
 we can introduce a partial order. That is, if a position $G$ is better than $H$ for Left whatever position is added to  both positions, then we say that $G$ is greater than $H$, and if $G$ is better than $H$ for Right, then we say that $G$ is smaller than $H$.

It is natural to have a question such that whether a similar partial order can be defined in $\LR$-ending partisan rulesets. However, as shown in Theorem \ref{thm:order}, there are no pairs of positions such that one is greater than the other.

\begin{definition}[Partial order]
\label{def:po}
The partial order of outcomes is $\oL > \oP > \oR, \oL > \oN > \oR, \oP \fzz \oN$. Here, $\oP \fzz \oN$ means that $\oP$ and $\oN$ are incomparable.

If positions $G$ and $H$ satisfy $o(G+X) \ge o(H+X)$ for any  position $X$, then $G \ge  H$. Similarly, if positions $G$ and $H$ satisfy $o(G+X) \le o(H+X)$ for any  position $X$, then $G \le  H$. 

We say $G > H$ (resp. $G < H, G \fzz H$) if and only if $G \ge H$ and $G \not \le H$ (resp. $G \le H$ and $G \not \ge H, $ $G \not \ge H$ and $G \not \le H$) hold.  
\end{definition}

Note that, from Definitions \ref{def:equiv} and \ref{def:po}, $G \eq H$ if and only if both $G \ge H$ and $G \le H$ hold.

\begin{theorem}
\label{thm:order}
    For any position $G$, there is no position $H$ such that $G > H$.
\end{theorem}

\begin{proof}
    Let $G, H,$ and $X$ be positions.
    Assume that $o(G + X) > o(H + X)$. Then either $o(G + X) = \oL$ or $o(H + X) = \oR$ holds. Let $X' \iso X + \sR$. We assume $o(G + X) = \oL$ and $o(H + X) \neq \oL$. Then, from Proposition \ref{prof:plussr}, $o(G + X') = \oR$ and $o(H + X') \neq \oR$. Thus, $o(G + X') < o(H+X').$ Similarly, if $o(G + X) \neq \oR$ and $o(H + X) = \oR$, then $o(G + X') \neq \oL$ and $o(H + X') = \oL$. Thus, $o(G + X') < o(H+X').$
    Therefore, for any position $G$, there is no position $H$ such that $G > H$.    
\end{proof}

\begin{corollary}
    For any positions $G$ and $H$, the relationship between them can be only $G \eq H$ or $G \fzz H$.
\end{corollary}

\section{Values and their constructions}
\label{sec:values}

\subsection{Notable values}
Some positions are regarded as equivalent through the equivalence relation. Thus, we can consider equivalence classes that are called values, and by examining their algebraic properties, combinatorial game theory efficiently investigates the outcome classes of sums of various positions. 

A famous example of equivalence classes in games is Sprague-Grundy values, investigated by Sprague in \cite{Spr} and Grundy in \cite{Gru}, for impartial games under the normal play convention.
A well-known equivalence class for partisan games under the normal play convention is Conway's game values \cite{C01}, which provide a rich algebraic structure generalizing the Sprague–Grundy theory.
Under the disjunctive sum, game values in normal play form a group.
\begin{remark}
    By Corollary \ref{cor:noinverse}, a position $G$ which satisfies $G = \conj{G}$ does not have an inverse in the game value monoid, and therefore this monoid is not a group.
\end{remark}

In impartial games under the normal play convention, there is only one terminal position $0$ and other positions are defined recursively like Definition \ref{def:positions}.
Then, for any position $G \iso \{G_1, \ldots, G_n\},$ its Sprague-Grundy value $\mathcal{G}(G)$ is recursively defined as

$$
\mathcal{G}(G) = {\rm mex}(\{\mathcal{G}(G_1), \ldots, \mathcal{G}(G_n)\}),
$$
where for $S \subsetneq \mathbb{Z}_{\ge 0}$, ${\rm mex}(S) = \min(\mathbb{Z}_{\ge 0} \setminus S)$.

Then, a position $G$ is a $\oP$-position if and only if $\mathcal{G}(G) = 0$ and $G$ is an $\oN$-position if and only if $\mathcal{G}(G) \ne 0$.
Moreover, $\mathcal{G}(G+H) = \mathcal{G}(G) \oplus \mathcal{G}(H),$ where $\oplus$ is the XOR operator for binary notation, which is also called {\em nim-sum} in the context of combinatorial game theory, since it is also used for analysis of a game called {\sc Nim} \cite{B02}.
This means that two positions whose Sprague-Grundy values are the same can be regarded as equivalent in impartial games under the normal play convention.

Regarding the nim-sum and ${\rm mex}$, we will use the following well-known result.
\begin{lemma}
\label{lem:mex}
    Let $a_1, \ldots, a_n$ be nonnegative integers. Then, $a_1 \oplus \cdots \oplus a_n = {\rm mex}(\bigcup_{1 \le i \le n}\{a_1 \oplus \cdots \oplus a_{i-1} \oplus a' \oplus a_{i+1} \oplus \cdots \oplus a_n \mid 0 \le a' < a_i\}).$
    
\end{lemma}

Note that if $a_1 \oplus \cdots \oplus a_n = a > 0$, then for any $0 \leq b < a$, there exists some $a_i$ such that $$b = a_1 \oplus \cdots \oplus a_{i-1} \oplus a' \oplus a_{i+1} \oplus \cdots \oplus a_n$$ for some $a' < a_i$.

As well as the Sprague-Grundy values, in this paper, we assign names to several characteristic values in $\LR$-ending partisan rulesets and show that their behavior is similar to Sprague-Grundy values.

\begin{definition}\label{def:sLsR}
    Let $\sL_0 \iso \sL$ and $\sR_0 \iso \sR$.
    For any positive integer $n$, we also let $\sL_n \iso \{\sL_{n-1}, \sL_{n-2}, \ldots, \sL_0\}$ and $\sR_n \iso \{\sR_{n-1}, \sR_{n-2}, \ldots, \sR_0\}$.
\end{definition}

Note that for any $n,$ $o(\sL_n) = \oL$ and $o(\sR_n) = \oR$ hold.

\begin{example}
    In $\LR$-ending partisan subtraction game with $S = 2\mathbb{Z}_{>0}$, the value of a position $(2n)$ is $\sL_n$. Moreover, the value of a position $(2n + 1)$ is $\sR_n$.
\end{example}

\begin{theorem}\label{mex}
    Assume that $G \iso \{\sL_{a_1}, \ldots, \sL_{a_\ell}\}$ and $a = {\rm mex}(\{a_1, \ldots, a_\ell\})$. Then, $G \eq \sL_{a}$.
\end{theorem}

\begin{proof}
    We show that for any position $X$, $o(G + X) = o(\sL_{a} + X).$ We prove this by induction on $X$. When $X$ is a terminal position, 
        $o(G + \sL) = o(\sL_{a}) = \oL$ and $o(G + \sR) = o(\sL_{a} + \sR) = \oR$.

    Assume that for any position $X$, if $b(X) < n$ then $o(G+X) = o(\sL_{a} + X).$ Consider the case $b(X)= n$.
    
    If Left has a winning strategy from $\sL_{a} + X$, then it is a move to $\sL_{a_i} + X \in \oL \cup \oP$, where $a_i < a$ or $\sL_{a} + X' \in \oL \cup \oP$ , where $X'$ is an option of $X$. 
    If $\sL_{a_i} + X \in \oL \cup \oP$, then $G + X$ also has an option to $\sL_{a_i} + X$. If $\sL_{a} + X' \in \oL \cup \oP$, from the induction hypothesis, $G + X' \in \oL \cup \oP$. Thus, Left can win from $G + X$ by moving to $\sL_{a_i} + X$ or $G + X'$.

    If Left can win from $\sL_{a} + X$ by playing second, for any $a_i < a, $ $\sL_{a_i} + X \in \oL \cup \oN$ and for any option $X'$ of $X$, $\sL_{a} + X' \in \oL \cup \oN$. From the induction hypothesis, $G + X' \in \oL \cup \oN$.

    The remaining case is Right moves to $\sL_{a_i} + X$, where $a_i > a$. In this case, from this position, Left can move to $\sL_{a} + X \in \oL \cup \oP$. Therefore, from every option of $G + X,$ Left can win by moving second, which means $G + X \in \oL \cup \oP$.

    Similarly, if Right can win from $\sL_a + X$ by moving first, Right can win from $G + X$ by moving first and if Right can win from $\sL_a + X$ by moving second, Right can win from $G+X$ by moving second.
    Therefore, $o(\sL_a+X) = o(G+X).$    
\end{proof}

\begin{corollary}\label{lrn}
    For any nonnegative integers $a_1$ and $a_2$, the following holds:
    \begin{itemize}
        \item $\sL_{a_1} + \sL_{a_2} \eq \sL_{a_1 \oplus a_2}.$
        \item $\sL_{a_1} + \sR_{a_2} \eq \sR_{a_1 \oplus a_2}.$
        \item $\sR_{a_1} + \sR_{a_2} \eq \sL_{a_1 \oplus a_2}.$
    \end{itemize}
\end{corollary}

\begin{proof}
    We prove this by induction on $a_1 + a_2$. When $a_1 = a_2 = 0, $ the statement is trivial by Axiom \ref{axm:base}.
    Consider the case $a_1 + a_2 > 0$. From the induction hypothesis, Theorem \ref{mex}, and Lemma \ref{lem:mex},
    \begin{eqnarray*}
    \sL_{a_1} + \sL_{a_2} &\iso& \{\sL_0 + \sL_{a_2}, \ldots, \sL_{a_1 -1 } + \sL_{a_2}, \sL_{a_1} + \sL_0,\\
    && \ldots,\sL_{a_1} + \sL_{a_2 - 1} \} \\
    & \eq & \{\sL_{0 \oplus a_2}, \ldots, \sL_{(a_1 - 1)\oplus a_2}, \sL_{a_1 \oplus 0}, \ldots, \sL_{a_1 \oplus (a_2 - 1)}\} \\
    & \eq & \sL_{{\rm mex}(\{0\oplus a_2, \ldots, (a_1 - 1) \oplus a_2, a_1 \oplus 0, \ldots, a_1 \oplus (a_2 - 1)\})} \\
    & \iso & \sL_{a_1 \oplus a_2}.
    \end{eqnarray*}

    The remaining $\sL_{a_1} + \sR_{a_2} \eq \sR_{a_1 \oplus a_2}$ and $\sR_{a_1} + \sR_{a_2} \eq \sL_{a_1 \oplus a_2}$ are shown in similar ways.    
\end{proof}

\begin{definition}\label{kome}
    For any positive integer $n$, 
    $$\kome_n \iso \{\sL_{n-1}, \sL_{n-2}, \ldots, \sL,\sR_{n-1}, \sR_{n-2}, \ldots,\sR\}.$$
\end{definition}

We give an example in which these values appear in Section \ref{even nim}.

When we consider a disjunctive sum of $\kome_n$, we need to calculate the nim-sum
 of their sizes, while ignoring the smallest two values.   
\begin{theorem}
\label{thm:kome}
For any positive integers $a_1, \ldots, a_n$ and nonnegative integers $b_1, \ldots, b_\ell, c_1, \ldots, c_r$,
    let $G \iso \kome_{a_1} + \cdots + \kome_{a_n} + \sL_{b_1} + \cdots + \sL_{b_\ell} + \sR_{c_1} + \cdots + \sR_{c_r}$ and assume that $a_1 \leq \cdots \leq a_n$. Then,
    $$
    o(G) = \left\{ \begin{array}{cc}
        \oL & (n = 0 \text{ and } r \text { is even}) \\
        \oR & (n = 0 \text { and } r \text{ is odd}) \\
        \oN & (n = 1, \text{ or } n \geq 2 \text { and } \psi(G) \neq 0)\\
        \oP & (n \geq 2 \text { and } \psi(G) = 0),
    \end{array}\right.
    $$
    where $\psi(G) = a_3 \oplus \cdots \oplus a_n \oplus b_1 \oplus \cdots \oplus b_\ell \oplus c_1 \oplus \cdots \oplus c_r$.
\end{theorem}
\begin{proof}
If $n = 0$, then the statement clearly holds. Thus, we assume that $n \geq 1$. If $n = 1$, then the first player wins by moving from $\kome_{a_1}$ to either $\sL$ or $\sR$. 
When $n \geq 2$, we prove the statement by induction on $n$. Furthermore, for the same $n$, we use induction on $a_1 + \cdots + a_n + b_1 + \cdots + b_\ell + c_1 + \cdots + c_r$.
If $n = 2$, we assume that $b_1 \oplus \dots \oplus b_\ell \oplus c_1 \oplus \dots \oplus c_r = 0$. If the first player moves one of $\sL_{b_1}, \dots, \sL_{b_\ell}, \sR_{c_1}, \dots, \sR_{c_r}$, then by Lemma \ref{lem:mex}, the value $b_1 \oplus \dots \oplus b_\ell \oplus c_1 \oplus \dots \oplus c_r$ changes, so it is no longer equal to $0$. Thus, by the induction hypothesis, the outcomes of these options are $\oN$. If the first player moves either $\kome_{a_1}$ or $\kome_{a_2}$, then by the result for the case $n = 1$, the outcomes of these options are $\oN$. Hence, if $b_1 \oplus \dots \oplus b_\ell \oplus c_1 \oplus \dots \oplus c_r = 0$, then the outcome of this position is $\oP$.
If $b_1 \oplus \dots \oplus b_\ell \oplus c_1 \oplus \dots \oplus c_r \neq 0$, then by Lemma \ref{lem:mex}, there exists a move that makes the nim-sum 0. 
Thus, the outcome of this position is $\oN$.
If $n > 2$, we assume that $a_3 \oplus \dots \oplus a_n \oplus b_1 \oplus \dots \oplus b_\ell \oplus c_1 \oplus \dots \oplus c_r = 0$. If the first player moves one of $\kome_{a_3}, \dots, \kome_{a_n}, \sL_{b_1}, \dots, \sL_{b_\ell}, \sR_{c_1}, \dots, \sR_{c_r}$, then by Lemma \ref{lem:mex}, the nim-sum changes, and is no longer equal to $0$. 
Thus, by the induction hypothesis, each of these options has outcome $\oN$. If the first player moves either $\kome_{a_1}$ or $\kome_{a_2}$, without loss of generality, we assume that the first player moves from $\kome_{a_2}$ to $\sL_j$. Since $j < a_2 \leq a_3$, moving from $\kome_{a_3}$ to $\sL_j$ is an available move. Thus, $j \oplus a_3 \neq 0$ and it follows that $j \oplus a_4 \oplus \dots \oplus a_n \oplus b_1 \oplus \dots \oplus b_\ell \oplus c_1 \oplus \dots \oplus c_r = (j \oplus a_3) \oplus a_3 \oplus a_4 \oplus \dots \oplus a_n \oplus b_1 \oplus \dots \oplus b_\ell \oplus c_1 \oplus \dots \oplus c_r \neq 0$. Hence, by the induction hypothesis, the outcome of this option is $\oN$. Therefore, if $a_3 \oplus \dots \oplus a_n \oplus b_1 \oplus \dots \oplus b_\ell \oplus c_1 \oplus \dots \oplus c_r = 0$, then the outcome of this position is $\oP$.
If $a_3 \oplus \dots \oplus a_n \oplus b_1 \oplus \dots \oplus b_\ell \oplus c_1 \oplus \dots \oplus c_r \neq 0$, then by Lemma \ref{lem:mex} and Definition \ref{kome}, there exists a move that makes the nim-sum 0. Therefore, the outcome of this position is $\oN$.
\end{proof}

We also have the following values, for which the outcome can be calculated using nim-sum.

\begin{definition}\label{stap}
$\stap 1 \iso \{\sL, \sR\}$ and 
    for any positive integer $n >1$, 
    $$\stap n \iso \{ \stap (n-1), \stap (n-2), \ldots, \stap 1, \sL, \sR  \}.$$
\end{definition}

Note that $\stap n = \overline{\stap n} = \stap n + \sR$.

We give an example in which these values appear in Section \ref{octal}.

\begin{theorem}\label{thm:stap}
For any positive integers $a_1, \ldots, a_n,$
    $\stap a_1 + \cdots + \stap a_n \in \oP$ if and only if $a_1 \oplus \cdots \oplus a_n = 0$. 
\end{theorem}
\begin{proof}
We proceed by induction on $n$. Furthermore, for the same $n$, we use induction on $a_1 + \cdots + a_n$. If $n = 1$, then the first player wins by moving to either $\sL$ or $\sR$.

We assume that $\stap a_1 + \cdots + \stap a_n \in \oP$ if and only if $a_1 \oplus \cdots \oplus a_n = 0$. 
If $a_1 \oplus \cdots \oplus a_n \oplus a_{n + 1} = 0$ and the first player moves from $\stap a_i$ to either $*L$ or $*R$, then, by Proposition \ref{prop:rconj}, 
we have $\stap a_1 + \cdots + *L + \cdots + \stap a_{n+1} = \stap a_1 + \cdots + *R + \cdots + \stap a_{n+1} $. By the induction hypothesis and Lemma \ref{lem:mex}, we obtain $\stap a_1 + \cdots + *L + \cdots + \stap a_{n+1} \in \oN$.
If $a_1 \oplus \cdots \oplus a_n \oplus a_{n + 1} = 0$ and the first player moves from $\stap a_i$ to $\stap a_i'$, then $\stap a_1 + \cdots + \stap a_i' + \cdots + \stap a_{n+1} \in \oN$, since, by Lemma \ref{lem:mex}, $a_1 \oplus \cdots \oplus a_i' \oplus \cdots \oplus a_n \oplus a_{n + 1} \neq 0$ holds.
Thus, if $a_1 \oplus \cdots \oplus a_n \oplus a_{n + 1} = 0$, then $\stap a_1 + \cdots + \stap a_n + \stap a_{n + 1} \in \oP$.

Conversely, if $a_1 \oplus \cdots \oplus a_n \oplus a_{n + 1} \neq 0$, then by  Lemma \ref{lem:mex}, there exists some $i$ such that either

     $\cdot$ moving from $\stap a_i$ to $\stap a_i'$ satisfies $a_1 \oplus \cdots \oplus a_i' \oplus \cdots \oplus a_n \oplus a_{n + 1} = 0$
     
or 

    $\cdot$ moving from $\stap a_i$ to $*L$ satisfies $a_1 \oplus \cdots \oplus a_{i - 1} \oplus a_{i + 1} \oplus \cdots \oplus a_n \oplus a_{n + 1} = 0$.

  It follows that if $a_1 \oplus \cdots \oplus a_n \oplus a_{n + 1} \neq 0$, then $\stap a_1 + \cdots + \stap a_n + \stap a_{n + 1} \in \oN$.  
\end{proof}

\begin{remark}
    Since $\stap n = \conj{\stap n}, \kome_n = \conj{\kome_n}$, from Proposition \ref{prop:np}, for any position $G$, $G + \stap n, G + \kome_n \in \oN \cup \oP$.
\end{remark}

\subsection{Simplification methods}
In combinatorial game theory, it is important to determine if a position can be replaced with another position that has a simpler structure. For example, in partisan games under the normal play convention, two simplification
 methods, ``removing dominated options'' and ``bypassing reversible options'',  are known.
For $\LR$-ending partisan rulesets, we show the following simplification
 methods.

\begin{theorem}\label{thm.dominate}
      Let $G = \{G_1, \ldots, G_n\}$, where $n \geq 1$. We assume that for any $G_j = \{G_{j1}, \ldots, G_{jm}\} (j \geq 2)$, there exists an option $G_{j\ell}$ satisfying $G_{j\ell} \eq \{G_1\}$. Then $G \eq \{G_1\}$ holds.
\end{theorem}

\begin{proof}
We prove that $o(G + X) = o(\{G_1\} + X)$ for all positions $X$. We first show that $G + X \in \oL \cup \oP \iff \{G_1\} + X \in \oL \cup \oP$.
If $X$ is a terminal position and $G + X \in \oL \cup \oP$, then for any option $G_j + X$, we have $G_j + X \in \oL \cup \oN$. In particular, $G_1 + X \in \oL \cup \oN$ holds. Therefore, any option for $\{G_1\} + X$ is in $\oL \cup \oN$, and hence $\{G_1\} + X \in \oL \cup \oP$.

Conversely, if $\{G_1\} + X \in \oL \cup \oP$, then $G_1 + X \in \oL \cup \oN$ holds. For $G + X$, note that for each $j = 2, \dots, n$, the option $G_j + X$ has a further option $G_{j\ell} + X$ satisfying $G_{j\ell} + X \eq \{G_1\} + X \in \oL \cup \oP$. Hence $G_j + X \in \oL \cup \oN$, and thus $G + X \in \oL \cup \oP$.

Now, we assume that for every option $X'$ of $X$, we have $o(G + X^\prime) = o(\{G_1\} + X^\prime)$.
   If $G + X \in \oL \cup \oP$, then since both $G_1 + X$ and $G + X'$ are options of $G + X$, we have $G_1 + X, G + X^\prime \in \oL \cup \oN$. Thus, $\{G_1\} + X \in \oL \cup \oP$ holds.

   Conversely, if $\{G_1\} + X \in \oL \cup \oP$, then $G_1 + X \in \oL \cup \oN$ holds. From the induction hypothesis, we have $G + X^\prime \eq \{G_1\} + X^\prime \in \oL \cup \oN$. For $G + X$, since any option $G_j + X, (j = 2, \dots, n)$ has an option $G_{j\ell} + X \eq \{G_1\} + X \in \oL \cup \oP$, we have $G_j + X \in \oL \cup \oN$. Therefore, $G + X \in \oL \cup \oP$ holds.
   
   By Proposition \ref{prof:plussr}, $G + X \in \oR \cup \oP$ if and only if $G + X + \sR  \in \oL \cup \oP$ holds. Furthermore, $\{G_1\} + X \in \oR \cup \oP$ if and only if $\{G_1\} + X + \sR  \in \oL \cup \oP$ holds. Hence, $G + X \in \oR \cup \oP$ if and only if $\{G_1\} + X \in \oR \cup \oP$.
  Therefore, $o(G + X) = o(\{G_1\} + X)$.   
\end{proof}

\begin{example}
    Since a position $\{\sR, \{\sL, \{\sR\}\}\}$ satisfies the assumption of Theorem \ref{thm.dominate}, we have $\{\sR, \{\sL, \{\sR\}\}\} \eq \{\sR\}$.
\end{example}

By Definition \ref{def:sum} and Corollary \ref{lrn}, we have $$\{\{\sL\}\} = \{\sL\} + \{\sL\}  = \sL_1 + \sL_1  \eq \sL_{1 \oplus 1} = \sL.$$ Similarly, we also have $\{\{\sR\}\} \eq \sR$. From these observations, we have the following useful theorem.

\begin{theorem}\label{thm.bypass}
    Let $G = \{G_1, \ldots, G_n\}$, where $n \geq 1$. 
    We assume that $G_1 \eq \{*L\}$ and that for any  $G_j = \{G_{j1}, \ldots, G_{jm}\}$, there exists an option $G_{j\ell}$ satisfying $G_{j\ell} \eq *L$. Then $G \eq *L$ holds.
    Similarly, the statement holds when $\sL$ is replaced by $\sR$.
\end{theorem}

\begin{proof}
    In Theorem \ref{thm.dominate}, by setting $G_1 = \{\sL\}$, the statement holds, since $\{\{\sL\}\}  \eq \sL$. Similarly, by setting $G_1 = \{\sR\}$, the statement also holds.
\end{proof}
Since $\{\sL, \sR\} + \{\sL\} \in \oN$ and $\{\{\{\sL, \sR\}\}\} + \{\sL\} \in \oP$ hold, we have $\{\sL, \sR\} \not \equiv \{\{\{\sL, \sR\}\}\}$. This example shows that if $o(G)=\oN$, then $\{\{G\}\} \not\equiv G$ may occur. Therefore, a further generalization of this theorem seems difficult.

\begin{example}\label{exp:bypass}
    Since a position $\{\{\sL\}, \{\sL, \sR\}\}$ satisfies the assumption of Theorem \ref{thm.bypass}, we have $\{\{\sL\}, \{\sL, \sR\}\} \eq \sL$.
\end{example}

\section{$\LR$-ending partisan rulesets and their values}
\label{sec:examples}

In this section, we introduce three

 ending partisan rulesets and show some analysis of the rulesets by using the values shown in Section \ref{sec:values}.
 Recall that a position with multiple piles is regarded as the disjunctive sum of the corresponding one-pile positions.

\subsection{{\sc Even Nim}}\label{even nim}
In {\sc Even Nim}, there are some piles of {\em even} tokens. Each player, in their turn, chooses a pile and removes some tokens. For every pile, when it is chosen for the first time in the game, any positive number of tokens may be removed. However, after that, when the pile is chosen again, an arbitrary {\em even} number of tokens must be removed.
When there are no legal moves, if the total number of tokens is even, Left wins, and if it is odd, Right wins.

\begin{theorem}
    Let $ (a_1, \ldots, a_n)$ be an initial position in {\sc Even Nim} such that there are $n$ piles whose numbers of tokens are $a_1, a_2, \ldots,$ and $a_n$. We assume that $a_1 \le \cdots \le a_n$. Note that $a_1, \ldots, a_n$ are positive even integers.

    Then the outcome of the initial position is as follows:

    $$
     \left \{
    \begin{array}{cc}
        \oN & (n = 1, \text{ or } n \geq 3\text { and } a_3 \oplus \cdots \oplus a_n  \neq 0) \\
        \oP & (n = 2, \text{ or } n \geq 3 \text { and } a_3 \oplus \cdots \oplus a_n  = 0).
    \end{array}
    \right .
    $$
    
\end{theorem}
\begin{proof}
    Let $G$ be a position in {\sc Even Nim}. Assume that in $G$, there is only one pile whose number of tokens is $a$. We show that the value of $G$ is 
    $$
    \left \{
    \begin{array}{cc}
        \sL & (G \text{ is initial and } a = 0) \\
        \kome_{\frac{a}{2}} & (G \text{ is initial and } a \ge 2) \\
        \sL_{\frac{a}{2}} & (G \text{ is not initial and $a$ is even}) \\
        \sR_{\frac{a- 1}{2}} &(G \text{ is not initial and $a$ is odd})
    \end{array}
    \right .
    $$
    We proceed by induction on $a$.
    For $G$, if $a = 0$, then it is a terminal position and Left wins. Thus, by definition, the value of $G$ is $\sL$.
    If $a = 1$, then $G$ is not initial and has no option, so Right wins. Thus, by definition, the value of $G$ is $\sR$.
    
    We assume that $G$ is neither initial nor terminal.
    If $a$ is even, then $G$ has $\frac{a}{2}$ options, and by the induction hypothesis, their values are $\sL, \sL_1, \dots, \sL_{\frac{a - 2}{2}}$.
    By Definition \ref{def:sLsR}, we have $$G \iso \{\sL, \sL_1, \dots, \sL_{\frac{a - 2}{2}}\} \eq \sL_{\frac{a}{2}}.$$
    Similarly, if $a$ is odd, then $G = \sR_{\frac{a - 1}{2}}$.
    
    Finally, if $G$ is initial and $a \geq 2$, then there are $a$ options and by the induction hypothesis, their values are $\sL, \sL_1, \dots, \sL_{\frac{a - 2}{2}}, \sR, \sR_1, \dots, \sR_{\frac{a - 2}{2}}$.
    It follows that 
    $$G = \{\sL, \sL_1, \dots, \sL_{\frac{a - 2}{2}}, \sR, \sR_1, \dots, \sR_{\frac{a - 2}{2}}\} = \kome_{\frac{a}{2}}.$$

    Therefore, the value of the position $(a_1, \ldots, a_n)$ in {\sc Even Nim} is
    $$
    \kome_{\frac{a_1}{2}} + \cdots + \kome_{\frac{a_n}{2}}
    $$
    and from Theorem \ref{thm:kome}, since $\frac{a_3}{2} \oplus \cdots \oplus \frac{a_n}{2} = 0$ if and only if $a_3 \oplus \cdots \oplus a_n = 0,$ the outcome is
        $$
     \left \{
    \begin{array}{cc}
        \oN & (n = 1, \text{ or } n \geq 3 \text { and } a_3 \oplus \cdots \oplus a_n  \neq 0) \\
        \oP & (n = 2, \text{ or } n \geq 3 \text { and } a_3 \oplus \cdots \oplus a_n  = 0).
    \end{array}
    \right .
    $$  
\end{proof}

\subsection{Octal Game $0.2\dot{3}$}\label{octal}
 Octal game $0.2\dot{3}$ can be considered as one of the cousins of NIM \cite{WW}.
 In the $\LR$-ending partisan version of this game, the game value $\stap n$ appears.
In this game, there are some piles of tokens. Each player, in their turn, chooses a pile and removes any positive number of tokens if the number of tokens in the chosen pile is not 1. When there are no legal moves, if the total number of tokens is even, Left wins, and if it is odd, Right wins.
\begin{theorem}
    Let $ (a_1, \ldots, a_n)$ be an initial position in the Octal game $0.2\dot{3}$ such that there are $n$ piles whose numbers of tokens are $a_1, a_2, \ldots,$ and $a_n$.  Note that $a_1, \ldots, a_n$ are nonnegative integers.

    Then the outcome of the initial position is as follows:
    $$
     \left \{
    \begin{array}{cc}
        \oL & (\mathrm{max}\{a_1,\dots, a_n\} \le 1 \text{ and } a_1 + \cdots + a_n \text{ is even}) \\
        \oR & (\mathrm{max}\{a_1,\dots, a_n\} \le 1 \text{ and } a_1 + \cdots + a_n \text{ is odd}) \\
        \oN & (\mathrm{max}\{a_1,\dots, a_n\} \ge 2 \text{ and } \bigoplus_{a_i \geq 2}(a_i - 1) \neq 0 )\\
        \oP & (\mathrm{max}\{a_1,\dots, a_n\} \ge 2 \text{ and } \bigoplus_{a_i \geq 2}(a_i - 1) = 0),
    \end{array}
    \right .
    $$
   where $\bigoplus_{a_i \geq 2}(a_i - 1)$ is the total nim-sum of $(a_i - 1)$ with $a_i \geq 2$.
\end{theorem}

\begin{proof}
    Let $G$ be a position in the Octal game $0.2\dot{3}$. Assume that in $G$, there is only one pile whose number of tokens is $a$. We show that the value of $G$ is 
    $$
    \left \{
    \begin{array}{cc}
        \sL & (a = 0) \\
        \sR & (a = 1) \\
        \stap (a - 1) & (a \ge 2).\\
    \end{array}
    \right .
    $$
    We proceed by induction on $a$.
    For $G$, if $a = 0 (\text{ or }a = 1)$, then it is a terminal position and Left( or Right) wins. Thus, by definition, the value of $G$ is $\sL(\text{ or }\sR)$.
    We assume that $a \ge 2$, that is, $G$ is not a terminal position. When $a \ge 2$, the options are $0, 1, 2, \dots, a-1$ and by the induction hypothesis, their values are $\sL, \sR, \stap 1, \dots, \stap (a - 2)$. By Definition \ref{stap}, we have
    $$G = \{\sL, \sR, \stap1, \dots, \stap (a - 2)\} = \stap(a - 1).$$
    Thus, the value of the position $(a_1, \dots, a_n)$ in the Octal game $0.2\dot{3}$ is
    $$\left \{
    \begin{array}{cc}
        \sL & (\mathrm{max}\{a_1,\dots, a_n\} \le 1 \text{ and } \sum_{i = 1}^n a_i  \text{ is even}) \\
        \sR & (\mathrm{max}\{a_1,\dots, a_n\} \le 1 \text{ and } \sum_{i = 1}^n a_i \text{ is odd}) \\
        \sum_{a_i \ge 2}  \stap (a_i - 1) & (\mathrm{max}\{a_1,\dots, a_n\} \ge 2),\\
    \end{array}
    \right .$$
    where $\sum_{a_i \ge 2}  \stap (a_i - 1)$ is the sum of game values $\stap (a_i - 1)$ with $a_i \geq 2$. 
    From Theorem \ref{thm:stap}, the outcome is
    $$\left \{
    \begin{array}{cc}
        \oL & (\mathrm{max}\{a_1,\dots, a_n\} \le 1 \text{ and } \sum_{i = 1}^n a_i \text{ is even}) \\
        \oR & (\mathrm{max}\{a_1,\dots, a_n\} \le 1 \text{ and } \sum_{i = 1}^n a_i \text{ is odd}) \\
        \oN & (\mathrm{max}\{a_1,\dots, a_n\} \ge 2 \text{ and } \bigoplus_{a_i \geq 2}(a_i - 1) \neq 0) \\
        \oP & (\mathrm{max}\{a_1,\dots, a_n\} \ge 2 \text{ and } \bigoplus_{a_i \geq 2}(a_i - 1) = 0).
    \end{array}
    \right .$$
    
\end{proof}

\subsection{$\LR$-ending partisan {\sc Subtraction Games}}
In subtraction games, there are some piles of tokens and a fixed set of removable numbers $S$. Each player, in their turn, chooses a pile and removes $s \in S$ tokens. 
In $\LR$-ending partisan subtraction games, when there are no legal moves, if the total number of tokens is even, Left wins, and if it is odd, Right wins.

\begin{theorem}\label{thm.25}
    Let $S = \{2, 5\}$. Assume that in $G$ there is only one pile whose number of tokens is $a$. The value of $G$ is given in the following  Table \ref{tbl:2_5}. The sequence of values exhibits a periodic structure with period 7.
    \begin{table}[h]
    \centering
    \caption{Values when $S = \{2, 5\}$.}
\label{tbl:2_5}
  \begin{tabular}{c|ccccccc}
     $a$& 0& 1& 2& 3& $4$& 5& 6  \\ \hline
    $0+$ & $*L$ &$*R$& $\{*L\}$& $\{*R\}$& $*L$& $\{*L, \{*R\}\}$ & $\{*L,*R\}$\\
    $7+$ &$*L$& $*R$& $\{*L\}$& $\{*R\}$& $*L$& $\{*L, \{*R\}\}$ & $\{*L,*R\}$
  \end{tabular}
\end{table}
\end{theorem}

\begin{proof}
    Let $a = 7k + \ell, k \in \mathbb{Z}_{\geq 0}, \ell = 0, 1, \dots, 6$.
    We proceed by induction on $k$. First, we consider the base case $k = 0$. If $\ell = 0, 1$, then their values are $\sL$ and $\sR$, since they are terminal positions. If $\ell = 2, 3$, then their options are $\sL$, $\sR$, respectively. Thus, their values are $\{\sL\}, \{\sR\}$, respectively. If $\ell = 4$, then its option is $\{\sL\}$. Hence, the value is $\{\{\sL\}\}$, which is reduced to $\sL$.
    If $\ell = 5$, then its options are $\sL$ and $\{\sR\}$. Thus, the value is $\{\sL, \{\sR\}\}$.
    If $\ell = 6$, then its options are $\sL$ and $\sR$, so the value is $\{\sL, \sR\}$.

    We assume that the statement holds until $k = i$, and we show that the statement still holds when $k = i + 1$. We assume that $k = i + 1$. If $\ell = 0$, then its options are $\{\sL, \{\sR\}\}$ and $\{\sL\}$. Thus, the value is $\{\{\sL\}, \{\sL, \{\sR\}\}\}$. By Theorem \ref{thm.bypass}, we have $\sL$. Similarly, if $\ell = 1$, then the value is $\sR$.
    If $\ell = 2$, then both options are $\sL$. Thus, the value is $\{\sL\}$.
    If $\ell = 3$, then its options are $\{\sL, \{\sR\}\}$ and $\sR$, so the value is $\{\sR, \{\sL, \{\sR\}\}\}$. By Theorem \ref{thm.dominate}, we have $\{\sR\}$.
    If $\ell = 4$, then its options are $\{\sL, \sR\}$ and $\{\sL\}$. Thus, the value is $\{\{\sL\}, \{\sL, \sR\}\}$. By Example \ref{exp:bypass}, we have $\sL$.
    For cases $\ell = 5$ and $\ell = 6$, their options are the same as for case $k = 0$. Hence, the statement still holds for the case $k = i + 1$.    
\end{proof}

Similarly, we can show the following statements.

\begin{theorem}\label{thm:4k1}
    Let $k$ be a positive integer and $S = \{2, 4k + 1\}$. Assume that in $G$ there is only one pile whose number of tokens is $a$. The value of $G$ is given in the following  Table \ref{tbl:2_4k1}. The sequence of values exhibits a periodic structure with period $4k + 3$.
    
    \begin{table*}[tb]
    \centering
    \caption{Values when $S = \{2, 4k + 1\} (\ell = 0, \dots, k - 1)$.}
\label{tbl:2_4k1}
  \begin{tabular}{c|ccccccc}
     $a$& $4\ell$& $4\ell + 1$& $4\ell + 2$& $4\ell + 3$& $4k$& $4k + 1$& $4k + 2$  \\ \hline
    $0+$ & $*L$ &$*R$& $\{*L\}$& $\{*R\}$& $*L$& $\{*L, \{*R\}\}$ & $\{*L,*R\}$\\
    $4k + 3+$ &$*L$& $*R$& $\{*L\}$& $\{*R\}$& $*L$& $\{*L, \{*R\}\}$ & $\{*L,*R\}$
  \end{tabular}
\end{table*}
\end{theorem}

\begin{theorem}\label{thm:4k3}
    Let $k$ be a positive integer and $S = \{2, 4k + 3\}$. Assume that in $G$ there is only one pile whose number of tokens is $a$. The value of $G$ is given in the following  Table \ref{tbl:2_4k3}. The sequence of values exhibits a periodic structure with period $4k + 5$.
    
    \begin{table*}[htb]
\centering
\caption{Values when $S = \{2, 4k + 3\}(\ell = 0, \dots, k - 1).$}
\label{tbl:2_4k3}
  \begin{tabular}{c|ccccccccc}
     $a$& $4\ell$& $4\ell + 1$& $4\ell + 2$& $4\ell + 3$& $4k$& $4k + 1$& $4k + 2$ & $4k + 3$ & $4k + 4$ \\ \hline
    $0+$ & $*L$ &$*R$& $\{*L\}$& $\{*R\}$& $*L$&$*R$& $\{*L\}$ & $\{*L,*R\}$& $\{*R, \{*L\}\}$\\
    $4k + 5+$ &$*L$& $*R$& $\{*L\}$& $\{*R\}$& $*L$& $*R$& $\{*L\}$& $\{*L,*R\}$& $\{*R, \{*L\}\}$
  \end{tabular}
\end{table*}
\end{theorem}

\begin{remark}
    Since the sub-monoids generated by the values appearing in Theorems  \ref{thm:4k1} and \ref{thm:4k3} coincide, we can compute the outcomes of multiple piles of the $\LR$-ending partisan subtraction games
    with mixed $S_k = \{2, 2k + 3\}$ for $k \in \mathbb{Z}_{\geq 1}$, varying depending on the piles, by determining the outcome of each element in the sub-monoid
    in the same manner as in the study of mis\`ere quotients \cite{PS}. Based on and referring to this paper, the details will appear elsewhere as \cite{HI}.
\end{remark}
    
\section{Open problems}
\label{sec:open}
In this paper, we introduced the concept of $\LR$-ending partisan rulesets and showed basic theorems,
as well as some values and their construction, which are extremely useful for analyzing $\LR$-ending partisan rulesets. Finally, we present several remaining open problems worthy of further investigation:

\begin{itemize}
    \item Considering other frameworks of ending partisan rulesets: Let us say that a ruleset is {\em ending partisan} if the options of positions for both players are the same and the winner is determined by which terminal position the game ends. Impartial games and $\LR$-ending partisan rulesets are only special cases of ending partisan rulesets and we can consider other frameworks. For example, the sets of options for both players are the same, and there are two kinds of terminal positions; If the game ends in one of the terminal positions, the next player of the turn wins and if the game ends in the other terminal position, the previous player wins.

    \begin{remark}
    By generalizing $\LR$-ending partisan rulesets, we can construct games with non-commutative, or even non-associative disjunctive sums, even within the framework of ending partisan rulesets. For example, suppose the terminal positions satisfy
$\sL+\sL=\sR+\sL=\sR+\sR=\sL,
\sL+\sR=\sR.$
This operation can be interpreted as the logical implication
$P\Rightarrow Q$, where $\sL$ corresponds to ``true'' and $\sR$
corresponds to ``false''.
Since implication is neither commutative nor associative,
the resulting disjunctive sum also fails to satisfy these properties. Details will be published elsewhere by the authors.
\end{remark}
    \item Characterizing other notable values: In Section \ref{sec:values}, we characterize three kinds of values. We can expect that there are other values with good structures, and it is meaningful to find them.    
    \item Determining whether every position has a canonical form:
    In partisan games under the normal play convention, it is known that from any position, if the two simplification methods, ``removing dominated options'' and ``bypassing reversible options'' are applied as many times as possible, one can reach a unique simplest position, which is called the canonical form of the original position. Is the same thing happening in $\LR$-ending partisan rulesets by using Theorems \ref{thm.dominate} and \ref{thm.bypass}? As a related question, we should also  consider whether there are other simplification methods than Theorems \ref{thm.dominate} and \ref{thm.bypass}.
    
\end{itemize}

\section*{Acknowledgments}

The authors thank anonymous reviewers for their careful reading and helpful suggestions.

This work is partially supported by The INOUE ENRYO Memorial Grant, TOYO University.

\end{document}